\documentclass[11pt]{amsart}
\usepackage{fullpage,url,amssymb,enumerate}
\usepackage[all,cmtip]{xy} 
\usepackage{mathrsfs} 
\usepackage{hyperref}
\usepackage[alphabetic,lite]{amsrefs} 

\DeclareFontEncoding{OT2}{}{} 

\DeclareMathOperator{\Min}{Min}
\DeclareMathOperator{\Cond}{Cond}
\DeclareMathOperator{\GL}{GL}
\DeclareMathOperator{\PGL}{PGL}
\DeclareMathOperator{\SL}{SL}
\DeclareMathOperator{\Aff}{Aff}
\DeclareMathOperator{\Rat}{Rat}
\DeclareMathOperator{\Res}{Res}
\DeclareMathOperator{\res}{res}

\DeclareMathOperator{\Gm}{\GG_m}

\newcommand{\adj}{\mathrm{adj}}

\renewcommand{\AA}{\mathbb{A}}

\newcommand{\GG}{\mathbb{G}}
\newcommand{\QQ}{\mathbb{Q}}
\newcommand{\PP}{\mathbb{P}}
\newcommand{\RR}{\mathbb{R}}
\newcommand{\ZZ}{\mathbb{Z}}

\newcommand{\calM}{\mathcal{M}}
\newcommand{\calO}{\mathcal{O}}

\newcommand{\fp}{\mathfrak{p}}
\newcommand{\fq}{\mathfrak{q}}

\newcommand{\mat}[1]{\begin{pmatrix}#1\end{pmatrix}}

\newtheorem{theorem}{Theorem}[section]
\newtheorem{lemma}[theorem]{Lemma}
\newtheorem{corollary}[theorem]{Corollary}
\newtheorem{proposition}[theorem]{Proposition}

\theoremstyle{definition}
\newtheorem{definition}[theorem]{Definition}
\newtheorem{question}[theorem]{Question}
\newtheorem{problem}[theorem]{Problem}
\newtheorem{conjecture}[theorem]{Conjecture}
\newtheorem{example}[theorem]{Example}

\newtheorem{algorithm}[theorem]{Algorithm}

\theoremstyle{remark}
\newtheorem{remark}[theorem]{Remark}

\newcommand{\IF}{\texttt{if}\;}
\newcommand{\RETURN}{\texttt{return}\;}
\newcommand{\FOR}{\texttt{for}\;}
\newcommand{\GOTO}{\texttt{goto}\;}
\newcommand{\tab}{\hspace*{2em}}
\newcommand{\In}{\mathrm{in}}
\newcommand{\total}{\mathrm{tot}}

\begin{document}
\title{Minimal models for rational functions in a dynamical setting}
\author{Nils Bruin}
\thanks{Research of first author supported by NSERC}
\address{Department of Mathematics, Simon Fraser University,
         Burnaby, BC V5A 1S6, Canada}
\email{nbruin@sfu.ca}
\urladdr{http://www.cecm.sfu.ca/\~{}nbruin}

\author{Alexander Molnar}
\address{Department of Mathematics, Queen's University,
         Kingston, ON K7L 3N6, Canada}
\email{a.molnar@queensu.ca}

\date{May 25, 2012}

\begin{abstract} We present a practical algorithm to compute models of rational functions with minimal resultant under conjugation by fractional linear transformations. We also report on a  search for rational functions of degrees $2$ and $3$ with rational coefficients that have many integers in a single orbit. We find several minimal quadratic rational functions with $8$ integers in an orbit and several minimal cubic rational functions with $10$ integers in an orbit. We also make some elementary observations on possibilities of an analogue of Szpiro's conjecture in a dynamical setting and on the structure of the set of minimal models for a given rational function.
\end{abstract}

\maketitle

\section{Introduction}

The results in this article are inspired by a conjecture by Silverman,
\begin{conjecture}\label{C:uniform_bound}
For each $d\geq 2$ there is a constant $C_d$ such that the following is true.
Let $\phi(z)\in\QQ(z)$ be a rational function of degree $d\geq 2$, such that $\phi^2$ is not a polynomial and for any $\alpha\in \QQ$, consider the \emph{orbit} of $\alpha$ under $\phi$, being
\[\calO_\phi(\alpha)=\{\alpha,\phi(\alpha),\phi(\phi(\alpha)),\ldots\}.\]
If $\phi$ is \emph{minimal} and $\calO_\phi(\alpha)$ is infinite as a set then
\[\#\{\beta\in \calO_\phi(\alpha): \beta\in\ZZ\} \leq C_d.\]
\end{conjecture}
The conjecture is a direct translation of a conjecture by Lang, inspired by work by Dem`janenko (\cite{lang:ec_da}*{page 140}), that the number of integral points on an elliptic curve in \emph{minimal} Weierstrass form is bounded above by a constant only depending on the field and the rank of the curve.

Both conjectures are ostensibly false if the \emph{minimal} condition is dropped. Silverman proposes the following definition for minimality of rational functions. Let $f,g\in\ZZ[z]$ be polynomials such that $\phi(z)=\frac{f(z)}{g(z)}$ and such that the coefficients of $f,g$ do not have a divisor in common. If $\deg(f)=\deg(g)$, see Section~\ref{S:preliminaries} for the full definition, we define
\[\Res(\phi)=|\res(f,g)|.\]

We have that the group of fractional linear transformations $\PGL_2(\QQ)=\{z\mapsto\frac{az+b}{cz+d}: a,b,c,d\in\QQ \text{ and } ad-bc\neq 0\}$ acts by conjugation on $\QQ(z)$, i.e., if $A\in\PGL_2(\QQ)$ then $\phi^A=A^{-1}\circ\phi\circ A$. Silverman considers the subgroup $\Aff_2(\QQ)=\{z\mapsto az+b: a,b\in\QQ\text{ and } a\neq 0\}$ and defines a rational function to be \emph{affine minimal} if
\[\Res(\phi)=\min\{\Res(\phi^A): A\in \Aff_2(\QQ)\}\]
and phrases Conjecture~\ref{C:uniform_bound} in terms of it.
Because $\ZZ$ is a Dedekind domain, this yields the same notion as full $\PGL_2(\QQ)$-minimality (see Proposition~\ref{P:GL_and_affine_minimal}).

In order to enable the gathering of experimental evidence for the conjecture, one obviously needs a procedure to decide if a given rational function $\phi(z)$ is (affine) minimal, analogous to Tate's Algorithm \cite{tate:algorithm} to compute minimal models of elliptic curves. The main contribution of this article is Algorithm~\ref{A:globalminimalmodel}, an explicit, practical procedure that, given a rational function $\phi$, tests whether it is minimal and if not, computes a fractional linear transformation $A$ such that $\phi^A$ is minimal. The procedure we describe applies to rational functions $\phi$ over any field $K$ that is the field of fractions of a principal ideal domain $R$. We also provide an implementation of the algorithm for rational functions over $\QQ$ in the computer algebra system Magma \cite{magma}, see \cite{brumol:electronic}.

We apply the algorithm as part of a search for minimal rational functions over $\QQ$ of degrees $2$ and $3$ with many integers in their orbits. We do this by prescribing an initial orbit consisting of small integers and interpolating the rational function $\phi$ through the prescribed values. We can then test if there are any more integers in the early part of the orbit and test if $\phi$ is minimal. A systematic search of possible initial orbits yielded among other results,
\[\dfrac{86z^2 - 1068z - 338}{z^2 + 7z - 338}\text{ with }\calO_\phi(0)=[0,\, 1,\, 4,\, 11,\, 12,\, 7,\,15,\,-374,\, \ldots]\]
and
\[\dfrac{7z^3 - 41z^2 - 216z + 180}{2z^3 - z^2 - 21z + 90}\text{ with }\calO_\phi(0)=[0,\,2,\,-6,\,6,\,-3,\,3,\,-9,\,5,\,-5,\,8,\ldots].\]
These are orbits with at least $8$ resp.\ $10$ integers in them, which is $2$ more than what one can prescribe using interpolation in either case. In particular we see that for Conjecture~\ref{C:uniform_bound} we would need at least $C_2\geq 8$ and $C_3\geq 10$. See Section~\ref{S:intpoints} and \cite{brumol:electronic} for the complete results of our search.

As an easy corollary of the construction of our algorithm, we see that if $f,g\in\ZZ[z]$ are monic polynomials with no roots in common and $2\deg(g)<\deg(f)+1$ then $\phi(z)=f(z)/g(z)$ is minimal (see Remark~\ref{R:trivial_minimal}). As a consequence, from
\[\phi(z)=\frac{z^d+p^r}{z},\]
we see that powers of primes occurring in resultants of minimal rational maps can be arbitrarily large. That means that a possible dynamical analogue of Szpiro's conjecture would require a more refined concept of conductor and/or of resultant than the most naive guesses, see Section~\ref{S:szpiro}.

Finally, we note that the set of minimal rational maps is the union of $\PGL_2(\ZZ)$-orbits. We show that at least for functions of odd degree, the set may consist of more than a single orbit (see Example~\ref{E:mult_orbit}). We make some remarks about the general structure in Section~\ref{S:mintran}. These remarks help us in providing Example~\ref{E:no_affine_minimal_model} of a rational function over $\QQ(\sqrt{-5})$ that does admit a minimal model but not via an affine transformation, thus providing an example that our algorithm is fundamentally restricted to principal ideal domains.

A significant part of the results in this paper come from the M.Sc.~thesis of the second author \cite{molnar:msc}.

\section{Preliminaries}\label{S:preliminaries}

Let $K$ be a field. Our main objects of study are rational morphisms
\[\phi\colon\PP^1\to\PP^1\]
of degree $d\geq 2$, defined over $K$. We follow the definitions and notation from \cite{ADS} and write  $\Rat_d(K)$ for the space of such rational morphisms. By choosing homogeneous coordinates $(X:Y)$ on $\PP^1$ we can represent a morphism $\phi$ by two homogeneous degree $d$ polynomials $F,G\in K[X,Y]$ such that
\[\phi(X:Y)=(F(X,Y):G(X,Y)).\]
We write $F(X,Y)=f_dX^d+f_{d-1}X^{d-1}Y+\cdots+f_0Y^d$ and $G(X,Y)=g_dX^d+g_{d-1}X^{d-1}Y+\cdots+g_0Y^d$.
It is often convenient to work with an affine coordinate $z=X/Y$ instead and write $f(z)=F(z,1)$ and $g(z)=G(z,1)$, so that we have
\[\phi(z)=\frac{f(z)}{g(z)}.\]

Rational morphisms $\phi$ defined over $K$ correspond to rational points on a quasi-projective variety $\Rat_d$ in the sense that the projective point $(f_d:\cdots:f_0:g_d:\cdots:g_0)\in\PP^{2d+2}(K)$ completely determines $\phi$. Let $\Res_d$ be the resultant of $F,G$ as degree $d$ forms. This is a bihomogeneous polynomial of bidegree $(d,d)$ in $f_0,\ldots,f_d$ and $g_0,\ldots,g_d$. In order for $\phi$ to be of degree $d$ we need that $\Res_d(F,G)$ does not vanish. Therefore, the variety $\Rat_d$ is the complement in $\PP^{2d+2}$ of the hypersurface $\Res_d=0$.

The automorphism group of $\PP^1$ is $\PGL_2$. It has a natural right-action on $\Rat_d$ via conjugation: for any $A\in\PGL_2$ we have $\phi^A=A^{-1}\circ\phi\circ A$. Rational maps in the same $\PGL_2$-orbit obviously have the same dynamical properties, so the appropriate moduli space for dynamical purposes is
\[\calM_d=\Rat_d/\PGL_2.\]

\begin{remark} See \cite{ADS}*{Section 4.4} for a discussion on its structure as an algebraic variety.
In general, there may be rational points on $\calM_d$ that do not have a rational point on $\Rat_d$ above them. These are rational morphisms for which the field of moduli is not equal to the field of definition. See \cite{ADS}*{Section 4.10} and \cite{sil:FOMFOD}.
\end{remark}

For our purposes it is more convenient to make a step in the other direction and consider the affine cone over $\Rat_d$. Given a rational morphism $\phi=(F:G)$, we say $[F,G]$ is a \emph{model} for $\phi$. Similarly, in affine coordinates, we have $\phi=f/g$ and we also write $[f,g]$ for the model of $\phi$, which encodes exactly the same information.

We also say it is a model for $[\phi]$, where $[\phi]$ is the class of $\phi$ in $\calM_d$. Naturally, if $[F,G]$ is a model for $\phi$ and $\lambda$ is a non-zero scalar, then $[\lambda F,\lambda G]$ is also a model for $\phi$.
We write $M_d$ for the space of models. The embedding
\[\begin{array}{ccc}
M_d&\to&\AA^{2d+2}\\
{}[F,G]&\mapsto&(f_d,\ldots,f_0,g_d,\ldots,g_0)
\end{array}\]
identifies $M_d$ with the affine open $\{\Res_d\neq 0\}\subset \AA^{2d+2}$.
We follow \cite{ADS}*{4.11} and lift the action of $\PGL_2$ on $\Rat_d$ to an action of $\GL_2$ on $M_d$ in a way that avoids division. For $A=\mat{\alpha&\beta\\\gamma&\delta}$ we consider the \emph{classical adjoint}
\[A^\adj=\det(A)\,A^{-1}=\mat{\delta&-\beta\\-\gamma&\alpha}.\]
Note that $[F,G]\in M_d$ and $A\in\GL_2$ can be interpreted as morphisms $\AA^2\to \AA^2$, so we can let $A$ act on $M_d$ by defining
\[[F,G]^A=[F_A,G_A]= A^{\adj}\circ[F,G]\circ A,\]
where
\begin{align*}
F_A(X,Y)&=\delta F(\alpha X+\beta Y, \gamma X + \delta Y)- \beta G(\alpha X+\beta Y, \gamma X + \delta Y)\\
G_A(X,Y)&=-\gamma F(\alpha X+\beta Y, \gamma X + \delta Y)+ \alpha G(\alpha X+\beta Y, \gamma X + \delta Y)
\end{align*}
It is easy to check that this action descends to the action of $\PGL_2$ on $\Rat_d$ we considered earlier.
We now have an action of $\Gm\times\GL_2$ on $M_d$ given by
\[[F,G]^{(\lambda,A)}=[\lambda F_A,\lambda G_A] \text{ where }(\lambda,A)\in \Gm\times\GL_2.\]
Furthermore, the compatibility with $\Rat_d$ gives us that $M_d/(\Gm\times\GL_2)=\calM_d$.

The main advantage of considering $M_d$ rather than $\Rat_d$ is that $\Res_d$ can be interpreted as a function on $M_d$. It is a covariant of the group we are considering.
\begin{proposition}\label{P:restrans}
Let $[F,G]\in M_d$ and let $(\lambda,A)\in\Gm\times \GL_2$. Then
\[\Res_d(\lambda F_A,\lambda G_A)=\lambda^{2d} \det(A)^{d^2+d}\Res(F,G).\]
\end{proposition}
\begin{proof}
See the proof of \cite{ADS}*{Proposition 4.95}.
\end{proof}

\begin{remark} Note that $\Res_d(F,G)$ is not equal to the univariate polynomial resultant $\res(f,g)$ if either $d_f=\deg_z(f)$ or $d_g=\deg_z(g)$ is smaller than $d$. We have the relation
\[\Res_d(F,G)=f_d^{d-d_g}((-1)^d g_d)^{d-d_f}\res(f,g).\]
\end{remark}

Now consider a field $K$ that is the field of fractions of an integral domain $R$. Let $[F,G]\in M_d(K)$ be a model of a rational morphism $\phi\in \Rat_d(K)$, and hence also a model of the isomorphism class $[\phi]\in \calM_d(K)$. We say that $[F,G]$ is a model \emph{over} $R$ if $F,G\in R[X,Y]$. By clearing denominators, one can always obtain a model over $R$ from a model over $K$. Note that if $[F,G]$ is a model over $R$ then $[F,G]\in\AA^{2d+2}(R)$, but that $[F,G]$ is an $R$-integral point on $M_d$ only if $\Res_d(F,G)$ is a unit in $R$.

\subsection{Minimal models}

\begin{definition}
Let $R$ be an integral domain with field of fractions $K$. Let $\phi\in\Rat_d(K)$.
We define the \emph{resultant} of $\phi$ to be the $R$-ideal generated by the resultants of the models of $\phi$ over $R$, i.e.,
\[\Res_R(\phi)=\big(\Res_d(F,G) : [F,G]\in M_d(K) \text{ and a model of $\phi$ over } R\big)R\]
Similarly, we define the \emph{resultant} of $[\phi]\in\calM_d(K)$ to be the $R$-ideal generated by the resultants of its models over $R$, i.e.,
\[\Res_R([\phi])=\big(\Res_d(F,G): [F,G]\in M_d(K)\text{ and a model of $[\phi]$ over } R\big)R\]
\end{definition}

\begin{remark} We do not concern ourselves with the resultants of classes in $\calM_d(K)$ that do not admit models over $K$.
\end{remark}

\begin{definition}
We say that $[F,G]\in M_d(K)\cap \AA^{2d+2}(R)$ is an $R$-\emph{minimal} model if $[F,G]$ is a model of $[\phi]$ with a resultant that generates the ideal $\Res_R([\phi])$, i.e.,
\[\Res_R([\phi])=\Res_d(F,G)R.\]
\end{definition}

\begin{definition} We write $\Aff_2\subset\GL_2$ for the algebraic subgroup of matrices that induce automorphisms of $\PP^1$ that leave $\infty=(1:0)$ invariant, i.e.,
\[\Aff_2(R)=\left\{\mat{\alpha&\beta\\0&\delta}\in\GL_2(R)\right\}\]
\end{definition}

The name is motivated by the fact that a matrix in $\Aff_2$ induces an affine transformation $z\mapsto \frac{1}{\delta}(\alpha z+\beta)$. We define $\calM_{d,1}=M_d/(\GG_m\times \Aff_2)$. For a rational map $\phi\in\Rat_d(K)$ we write $[\phi]_1\in \calM_{d,1}(K)$. We say that $[F,G]\in M_d(K)$ is a \emph{model} for $[\phi]_1$ if $[F/G]_1=[\phi]_1$ (i.e., if there is an affine transformation that conjugates one into the other). We define
\[\Res_R([\phi]_1)=\big(\Res_d(F,G): [F,G]\in M_d(K)\text{ and a model of $[\phi]_1$ over } R\big)R.\] 

\begin{definition}
We say that $[F,G]\in M_d(K)\cap \AA^{2d+2}(R)$ is an $R$-\emph{affine minimal} model if $[F,G]$ is a model of $[\phi]_1$ with a resultant that generates $\Res_R([\phi]_1)$, i.e.,
\[\Res_R([\phi]_1)=\Res_d(F,G)R.\]
\end{definition}

\begin{proposition}\label{P:aff_decomp}
Let $R$ be a principal ideal domain with field of fractions $K$. Then
\[\GL_2(K)=\Aff_2(K)\SL_2(R)\text{ and }\GL_2(K)=\SL_2(R)\Aff_2(K).\]
\end{proposition}
\begin{proof}
Let $B=\mat{\alpha&\beta\\\gamma&\delta}\in\GL_2(K)$. In order to establish the first claim we exhibit a matrix $C\in\SL_2(R)$ such that $BC\in\Aff_2(K)$. If $\gamma=0$ we can take $C$ to be the identity matrix. Otherwise, there are coprime $a,c\in R$ such that $\frac{\delta}{\gamma}=-\frac{a}{c}$.
It follows that $a\gamma+c\delta=0$ and that there are $b,d\in R$ such that $ad-bc=1$. We can take $C=\mat{a&b\\c&d}\in\SL_2(R)$. The second claim follows by an analogous argument.
\end{proof}

\begin{proposition}\label{P:GL_and_affine_minimal}
Let $R$ be a Dedekind domain with field of fractions $K$. Then for any $\phi\in \Rat_d(K)$ we have $\Res_R([\phi])=\Res_R([\phi]_1)$. In particular, a model $[F,G]$ for $\phi$ is $R$-affine minimal if and only if it is $R$-minimal.
\end{proposition}

\begin{proof}
First note that Proposition~\ref{P:restrans} establishes that $\Res_d$ is invariant under $\SL_2$, so Proposition~\ref{P:aff_decomp} immediately gives the result for principal ideal domains $R$.

If $R$ is a Dedekind domain, it is straightforward to check that a model is $R$-(affine) minimal if and only if it is $R_\fp$-(affine) minimal for all localizations $R_\fp$ at primes $\fp$. Furthermore, for Dedekind domains, the localizations $R_\fp$ are principal ideal domains, so locally, minimality and affine minimality coincide. More explicitly, one checks that
$\Res_{R_\fp}([\phi])=\Res_R([\phi])R_\fp$ and that $\Res_{R_\fp}([\phi]_1)=\Res_R([\phi]_1)R_\fp$ and that
$I,J\subset R$ are equal if and only if for all primes $\fp$ we have $IR_\fp=JR_\fp$.
\end{proof}

\begin{remark} Silverman \cite{ADS}*{Proposition~4.100} shows that if $R$ is a Dedekind domain with a non-trivial class group, then not every class $[\phi]$ admits an $R$-minimal model. As we will see in Corollary~\ref{C:minimal_model}, if $R$ is a principal ideal domain, then any class admits an $R$-minimal model. In fact, Proposition~\ref{P:aff_decomp} implies that such a model can be obtained from any given model via an affine transformation.

Note that Proposition~\ref{P:GL_and_affine_minimal} does \emph{not} imply this in general: if $R$ has a non-trivial ideal class group, then it is possible to have a rational function $\phi$ such that $[\phi]$ admits an $R$-minimal model, but $[\phi]_1$ does not admit an $R$-affine minimal model. See Example~\ref{E:no_affine_minimal_model}.
\end{remark}

\subsection{Minimal models over discrete valuation rings}

We now restrict to the case where $R$ is a discrete valuation ring, with maximal ideal $\fp$, field of fractions $K$ and valuation $v\colon K\to\ZZ\cup\{\infty\}$. We write
\[v\mat{\alpha&\beta\\ \gamma&\delta}=\min( v(\alpha),\ldots,v(\delta) )\]
as well as
\[v(\sum_{i=0}^d f_i z^i)=\min(f_0,\ldots,f_d) \text{ and } v([F,G])=\min(v(F),v(G)).\]
With these definitions it is easy to check that for $[F,G]\in M_d(K)$ and $(\lambda,A)\in (\GG_m\times\GL_2)(K)$, there is a bound $B$ such that for any $(\lambda',A')$ such that $v(\lambda-\lambda')>B$ and $v(A-A')>B$, we have $v(\Res_d(\lambda F_A,\lambda G_A))=v(\Res_d(\lambda' F_{A'},\lambda' G_{A'}))$.

\begin{proposition}\label{P:localminimal}
Let $R$ be a discrete valuation ring with field of fractions $K$ and uniformizer $\pi$.
Let $\phi\in\Rat_d(K)$ be a rational function given by a model $[F,G]\in M_d(K)$. Then there are $e_1,e_2,e_3\in\ZZ$ and $\beta\in K$ such that for any $\beta'\in\beta+\pi^{e_3}R$ we can set
\[(\lambda,A)=(\pi^{e_1},\mat{\pi^{e_2}&\beta'\\0&1})\in (\GG_m\times\GL_2)(K)\]
and have that $[\lambda F_A,\lambda G_A]$ is an $R$-minimal model for $\phi$.
\end{proposition}

\begin{proof}
Since $R$ is a discrete valuation ring, we know that $\inf\{v(a):a\in \Res_R([\phi])\}$ is attained in the ideal and the triangle inequality shows it must be attained by the resultant of a model over $R$. This shows that there is a minimal model. In fact, we can use the same reasoning to assert the existence of an affine minimal model for $[\phi]_1$ and Proposition~\ref{P:GL_and_affine_minimal} guarantees that this model is also minimal. This shows that we can attain a minimal model by a transformation $(\lambda,A)\in (\GG_m\times \Aff_2)(K)$. It remains to prove that we can restrict to a transformation of the shape described.

First note that $(\lambda \delta^{d+1},\mat{\frac{\alpha}{\delta}&\frac{\beta}{\delta}\\0&1})$ and $(\lambda,\mat{\alpha&\beta\\0&\delta})$ have the same effect, so we can assume that $\delta=1$. Next note that transforming by $(\GG_m\times\GL_2)(R)$ does not change minimality, so we can assume that $\lambda$ and $\alpha$ are powers of a given uniformizer.

It remains to show that $v(\Res_d(\lambda F_A,\lambda G_A))$ remains constant under small perturbations of $\beta$. Since the resultant is polynomial in $\beta$, its valuation is locally constant away from zero and the desired result follows.
\end{proof}

\begin{corollary}\label{C:minimal_model}
Let $R$ be a principal ideal domain with field of fractions $K$. Then for any $\phi\in\Rat_d(K)$, the class $[\phi]\in\calM_d(K)$  has an $R$-minimal model $[F,G]$.
\end{corollary}

\begin{proof}
First let $[F,G]$ be any model of $\phi$ over $R$. Since $R$ is a Dedekind domain we have the factorisation $\Res_d(F,G)R=\fp_1^{e_1}\cdots \fp_n^{e_n}$ into prime ideals. It follows that $[F,R]$ is $R_\fq$-minimal for all primes $\fq\notin\{\fp_1,\ldots,\fp_n\}$.

We modify $[F,G]$ iteratively to ensure minimality for each index $i$ in the following way.
The assumption that $R$ is a principal ideal domain ensures that $\fp_i=\pi_iR$ for some $\pi_i\in R$.
We apply Proposition~\ref{P:localminimal} to find a transformation $(\lambda,A)$ such that $[F,G]^{(\lambda,A)}$ is $R_{\fp_i}$-minimal.
Since $R$ is dense in the localization $R_{\fp_i}$, we can choose $\beta'\in \pi_i^{e_4}R$ for some $e_4\in\ZZ$. This means that $(\lambda,A)\in(\GG_m\times\GL_2)(R_\fq)$ for any prime $\fq\neq\fp_i$ and hence that $[F,G]^{(\lambda,A)}$ is minimal at $\fp_i$ as well as at all primes where $[F,G]$ is already minimal. By iteratively applying such a transformation for each $i=1,\ldots,n$, we obtain a model that is minimal locally at all primes and hence is $R$-minimal.
\end{proof}

\section{Determining local minimal models}
\label{S:local}

Let $R$ be a discrete valuation ring with maximal ideal $\fp$, field of fractions $K$, uniformizer $\pi$ and valuation $v\colon K\to \ZZ\cup\{\infty\}$. We write $k=R/\fp$ for the residue field.

Let $\phi\in\Rat_d(K)$ be a rational function given by a model $[F,G]$ over $R$. In this section we develop a relatively efficient algorithm to compute a transformation
\begin{equation}\label{E:loctrans}
(\lambda,A)=(\pi^{e_1},\mat{\pi^{e_2}&\beta\\0&1})\in (\GG_m\times\GL_2)(K)
\end{equation}
of the form described in Proposition~\ref{P:localminimal}, such that $[\lambda F_A,\lambda G_A]$ is an $R$-minimal model of $[\phi]\in\calM_d(K)$. 
We do this by formulating a procedure that finds $e_1,e_2\in\ZZ$ and $\beta\in K$, or shows they do not exist, such that $\lambda F_A,\lambda G_A\in R[X,Y]$ and $v(\Res_d(\lambda F_A,\lambda G_A))<v(\Res_d(F,G))$.
First we observe a case where it is particularly easy to recognise that a model is minimal.
\begin{lemma}\label{L:trivial_minimal}
If $d$ is even and $v(\Res_d(F,G))<d$ or if $d$ is odd and $v(\Res_d(F,G))<2d$ then $[F,G]$ is an $R$-minimal model for $[\phi]$.
\end{lemma}
\begin{proof}
Proposition~\ref{P:restrans} shows that transformations can only change the resultant by $\gcd(2d,d^2+d)$-th powers. Since a minimal model has $\Res_d(F,G)\in R$, it must have non-negative valuation. Therefore, if the valuation is already small enough, a transformation cannot reduce it and keep the model over $R$.
\end{proof}

If we do find such values, we repeat the procedure with the transformed model; otherwise we have shown that the original model is minimal.  In light of Proposition~\ref{P:restrans}, we need
\[2d e_1 +(d^2+d)e_2 < 0.\]
Without loss of generality we can take
\begin{equation}\label{E:e1eq}
e_1=-\min(v(F_A),v(G_A)).
\end{equation}

We write
\[F_A=\sum f_i' X^iY^{d-i}\text{ and }G_A= \sum g_i' X^iY^{d-i}.\]
It follows that
\begin{equation}\label{E:f_g_prime}
\begin{aligned}
f_j'=f_j'(e_2,\beta)&=\pi^{je_2}\sum_{i=j}^d\binom{i}{j}(f_i\beta^{i-j}-g_i\beta^{i-j+1}),\\
g_j'=g_j'(e_2,\beta)&=\pi^{(j+1)e_2}\sum_{i=j}^d\binom{i}{j}g_i\beta^{i-j}.
\end{aligned}
\end{equation}
Finding a valuation-reducing transformation amounts to finding $e_2\in\ZZ$ and $\beta\in K$ such that
\begin{equation}\label{E:ineqs}
v(f_i')> \frac{d+1}{2}e_2\quad\text{ and }\quad v(g_i')> \frac{d+1}{2}e_2\quad\text{ for all }i=0,\ldots,d.
 \end{equation}
We proceed by proving lower and upper bounds for $e_2$ given $F,G$ and then lower bounds on $v(\beta)$ given $e_2,F,G$.

\begin{lemma}\label{L:e2_upper_bound_helper}
Let $f,g\in R[z]$ of degrees at most $d$. Then for any $\beta\in K$ we have
\[\min(v(f(\beta)),v(g(\beta)))\leq v(\res(f,g)).\]
\end{lemma}

\begin{proof} We first consider the case $v(\beta)\leq 0$.
Usual properties for resultants (see e.g.~\cite{ADS}*{Proposition~2.13c}; the proof there is stated for $R=\ZZ$, but is valid for arbitrary commutative rings) yield polynomials $U(z),V(z)\in R[z]$ of degree at most $d-1$ such that
\[Uf+Vg=z^{2d-1}\res(f,g).\]
In particular, we find that
\[v(\res(f,g))+(2d-1)v(\beta)\geq\min(v(U(\beta))+v(f(\beta)),v(V(\beta))+v(g(\beta))).\]
Since we have $v(U(\beta)),v(V(\beta))\geq (d-1)v(\beta)$, this yields
\[\min(v(f(\beta)),v(g(\beta)))\leq v(\res(f,g))+dv(\beta)\leq v(\res(f,g)).\]

For the case $v(\beta)\geq 0$ we use (see again e.g.~\cite{ADS}*{Proposition~2.13c}) that there are polynomials $U,V\in R[z]$ of degree at most $d-1$ such that
\[Uf+Vg=\res(f,g).\]
We have
\[\min(v(U(\beta))+v(f(\beta)),v(V(\beta))+v(g(\beta)))\leq v(\res(f,g)),\]
and since $v(U(\beta)),v(V(\beta))\geq 0$, the statement follows.
\end{proof}

\begin{lemma}\label{L:e2_bounds}
Let $[F,G]\in M_d(K)$ be a model over $R$. Let $f(z)=F(z,1)$ and $g(z)=G(z,1)$.
Let $d_G$ be the degree of $g$. Suppose $e_2\in\ZZ$ and $\beta\in K$ provide a solution to \eqref{E:ineqs}. Then we have
\[e_2>\begin{cases}
       -\frac{2}{2d_G-d+1} v(g_{d_G})&\text{if } d_G>\frac{1}{2}(d+1),\\
       -\frac{2}{d-1}v(f_d)&\text{if }d_G<d.
      \end{cases}
\]
Furthermore, we have
\[
  e_2 < \frac{2}{d-1}v(\res(f-zg,g)) = \begin{cases}
                           \frac{2}{d-1}v(\res(f,g))&\text{if } d_G<d,\\
                           \frac{2}{d-1}(v(\res(f,g))+v(g_d))&\text{if }d_G=d.
                          \end{cases}
\]
\end{lemma}

\begin{proof}
We use the notation $f_i,g_i,f'_i,g'_i$ as defined in \eqref{E:f_g_prime} and earlier.

We first prove the lower bounds. 
If $d_G>\frac{1}{2}(d+1)$, we consider $g'_{d_G}(e_2,\beta)=\pi^{(d_G+1)e_2}g_{d_G}$. Its valuation combined with \eqref{E:ineqs} gives the bound stated. If $d_G<d$ we have that $f'_d=\pi^{de_2}f_d$ and that $f_0\neq0$. Its valuation combined with \eqref{E:ineqs} yields the bound stated.

For the upper bound, we consider \eqref{E:ineqs} for $f'_0$ and $g'_0$. They yield
\[
v(f(\beta)-\beta g(\beta)) > \frac{d+1}{2}e_2\quad\text{ and }\quad v(g(\beta)) > \frac{d-1}{2}e_2.
\]
From Lemma~\ref{L:e2_upper_bound_helper} we obtain an upper bound on the minimum of the left hand sides of the inequalities, which leads immediately to the upper bound stated in the lemma. It is a straightforward exercise in Sylvester matrices to see that $\res(f-zg,g)=\res(f,g)$ if $\deg(g)<\deg(f)$ and that $\res(f-zg,g)=\pm g_d \res(f,g)$ if $\deg(f)\leq \deg(g)=d$. In either case this provides a finite upper bound, because $f,g$ are coprime.
\end{proof}

\begin{remark}\label{R:trivial_minimal}
Note that the argument that provides the lower bound for $e_2$ if $d_G>\frac{1}{2}(d+1)$ gives the upper bound $e_2<\frac{2}{d-1-2d_G}v(g_{d_G})$ if $d_G<\frac{1}{2}(d+1)$. In particular, if $f,g\in R[z]$ are monic and $\deg(g)\leq \frac{1}{2}\deg(f)$ then we have $v(g_{d_G})=v(f_d)=0$ and we see that a solution to \eqref{E:ineqs} would require both $e_2>0$ and $e_2<0$. It follows that the corresponding model $[F,G]$ is already a minimal model for $[f/g]$. 
\end{remark}

\begin{remark}
For obtaining the upper bound on $e_2$ we considered the inequalities in \eqref{E:ineqs} arising from $f_0'$ and $g_0'$, because those are guaranteed to provide a finite upper bound. However, \eqref{E:ineqs} gives rise to multiple inequalities
\begin{align*}
v\left(\sum_{i=j}^d\binom{i}{j}(f_i\beta^{i-j}-g_i\beta^{i-j+1})\right)>\frac{d+1-2j}{2}e_2\\
v\left(\sum_{i=j}^d\binom{i}{j}g_i\beta^{i-j}\right)>\frac{d-1-2j}{2}e_2,
\end{align*}
so applying Lemma~\ref{L:e2_upper_bound_helper} on any pair of them (with $2j< d+1$ resp.~$2j<d-1$) potentially yields a sharper upper bound on $e_2$.
\end{remark}

With Lemma~\ref{L:e2_bounds} we have restricted the possible $e_2$ to a finite set. For each possible $e_2$, we are left with determining a value $\beta\in K$ that satisfies \eqref{E:ineqs}. Note that $f'_j, g'_j$ are polynomial in $\beta$, so  after clearing denominators, we obtain a problem of the following form.

\begin{problem}\label{P:find_beta}
Given $\{(h_1,c_1),\ldots,(h_r,c_c)\}$ with
\[
h_1,\ldots,h_r\in R[z]\quad\text{ and }\quad c_1,\ldots,c_r\in \RR,
\]
determine $\beta\in K$ such that
\[ v(h_i(\beta)) >c_i\text{ for }i=1,\ldots,r,\]
or prove that no such $\beta$ exists.
\end{problem}

\begin{lemma}\label{L:beta_lower_bound}
Let $f=\sum_{i=0}^n f_i z^i\in R[z]$ be a polynomial of degree $n$. Let
\[B(f,c)=\min\left(\frac{c-v(f_n)}{n},\min\left\{\frac{v(f_i)-v(f_n)}{n-i}: i =0,\ldots,n-1\right\}\right),\]
then for any $\beta\in K$ such that $v(f(\beta))>c$ we have
$v(\beta)\geq B(f,c).$
\end{lemma}

\begin{proof}
We observe that if $v(f(\beta))>c$ then we must have $v(f_n\beta^n)>c$ or $v(f_n\beta^n)\geq v(f_i\beta^i)$ for some $i=0,\ldots,n-1$. Solving for $v(\beta)$ provides the bound stated.
\end{proof}

Using Lemma~\ref{L:beta_lower_bound} we see that if $\beta$ is a solution for Problem~\ref{P:find_beta} and $B=\max\{B(h_i,c_i): i = 1,\ldots,r\}$, then $\beta=\pi^{-B}\beta'$ for some $\beta'\in R$, which itself is a solution to the problem
\begin{equation}\label{E:V}
V=\begin{cases}
\{(\pi^{\deg(h_i)B}h_i(\pi^{-B}z),c_i+B):i = 1,\ldots,r\}&\text{ if }B> 0,\\
\{(h_i(\pi^{-B}z),c_i):i = 1,\ldots,r\}&\text{ if }B\leq 0.
\end{cases}
\end{equation}
Because we have now reduced the problem to find a solution $\beta\in R$, we can use reduction. For $\beta\in R$ we write $\overline{\beta}$ for its residue class in $k$ and for $h\in R[z]$ we write $\overline{h}\in k[z]$ for its coefficient-wise reduction.
We obtain the following algorithm.

\begin{algorithm}\label{A:inequalitysolutions}
$\mathsf{InequalitySolutions}(V)$
\begin{description}
\itemsep=0.2em
\item[Input] $V=\{(h_1,c_1),\ldots,(h_r,c_r)\}\subset R[z]\times\RR$
\item[Output] An element $\beta\in R$ such that $v(h_i(\beta))>c_i$ for $i=1,\ldots,r$ or \textsf{none} if no such solution exists. 
\end{description}
\medskip
\begin{enumerate}[\hspace{2em}1)\;]
\itemsep=0.2em
 \item $V':=\{(\pi^{-v(h_i)}h_i,c_i-v(h_i))\text{ for those } i=1,\ldots,r \text{ for which }h_i\neq 0\text{ and } c_i\geq v(h_i)\}$.
 \item \IF $V'=\emptyset$: \RETURN $0$.
 \item $\overline{g}:=\gcd(\overline{h'}_i: (h'_i,c'_i)\in V')$.
 \item Let $W\subset R$ be a set of representatives of the roots of $\overline{g}(z)$ in $k$.
 \item \FOR $\beta_0\in W$:
 \item \tab $V'':=\{(\pi^{-1}h'_i(\beta_0+\pi z),c'_i-1): (h'_i,c'_i)\in V'\}$,
 \item \tab $\beta_1:=\textsf{InequalitySolutions}(V'')$,
 \item \tab \IF $\beta_1 \neq \mathsf{none}$: \RETURN $\beta_0+\pi\beta_1$.
 \item \IF $W=\emptyset$ or $\beta_1 = \mathsf{none}$ for all $\beta_0\in W$: \RETURN \textsf{none}.
\end{enumerate}
\end{algorithm}
Since the algorithm is recursive, we need to argue it will finish in finite time. The valuation bounds in $V''$ are decreased by at least $1$ from the ones that occur in $V$. Furthermore, note that any conditions with a negative valuation bound get removed in step 1) and that the algorithm terminates if $V'=\emptyset$. This means that $\max c_i$ is a bound on the recursion depth of the algorithm.

Furthermore, note that the polynomials in $V'$ all have non-zero reduction, so $\overline{g}$ computed in 3) is well defined. That means that $W$ in step 4) is a finite set, so the loop in 5) is finite. This establishes that the algorithm finishes in finite time.

For correctness, first note that in 1) we ensure that the polynomials in $V'$ have integral coefficients and that at least one of them is a unit in $R$ and that all vacuous conditions are removed from $V'$. If no conditions remain, then any $\beta\in R$ is a valid solution, so if a value is returned in step 2), it is correct.

Furthermore, it is clear that any solution would have to reduce to a root of $h'_i$ in $k$, for all $i$.
This means that $\beta=\beta_0+\pi\beta_1$, where $\beta_0$ represents such a root and $\beta_1\in R$, where $\beta_1$ satisfies the conditions represented by $V''$. If we find such a $\beta_1$ in step 7), we return the resulting solution in step 8). On the other hand, if we cannot find a suitable $\beta_1$ for any of the $\beta_0$, we have shown that no solutions exist. Note that the set $W$ can be empty, in which case there are no $\beta_0$ to try and \textsf{none} is returned immediately in step 9).

An algorithm to compute an $R$-minimal model for $[\phi]\in\calM_d(K)$ given by a model $[f,g]\in M_d(K)$ is now a matter of bookkeeping.

\begin{algorithm}\label{A:localminimalmodel}
$\mathsf{LocalMinimalModel}(f_\In,g_\In)$
\begin{description}
\itemsep=0.2em
\item[Input] $f_\In,g_\In\in R[z]$ with $\max(\deg(f_\In),\deg(g_\In))=d$ and $\phi=f_\In/g_\In\in \Rat_d(K)$.
\item[Output] $e_{1,\total},e_{2,\total}\in \ZZ$ and $\beta_\total\in K$ describing a transformation $(\lambda,A)$ as in Equations \eqref{E:loctrans} and \eqref{E:e1eq} and $f,g\in R[z]$ such that $[f,g]=[f_\In,g_\In]^{(\lambda,A)}$ is a minimal model for $[\phi]\in\calM_d(K)$.

If $[f_\In,g_\In]$ is already minimal then $[f,g]=[f_\In,g_\In]$ and $(e_1,e_2,\beta)=(0,0,0)$.
\end{description}
\medskip
\begin{enumerate}[\hspace{2em}1)\;]
 \itemsep=0.2em
\item $e_{1,\total},e_{2,\total},\beta_\total:=0,0,0$ and $f,g:=f_\In,g_\In$.
\item $e_1:=-\min(v(f),v(g))$; $e_{1,\total}:=e_{1,\total}+e_1$; $f:=\pi^{e_1} f$; $g:=\pi^{e_1} g$.
\item \FOR $e_2$ in the range given by Lemma~\ref{L:e2_bounds}:
\item \tab $V':=\{(f'_i,\frac{d+1}{2})\}\cup\{(g'_i,\frac{d+1}{2})\}$ as in Equation~\eqref{E:ineqs}.
\item \tab Let $V$ be as in Equation~\eqref{E:V}, where $B:=\max\{ B(h_i,c_i) : (h_i,c_i)\in V'\}$.
\item \tab $\beta':=\mathsf{InequalitySolutions}(V)$.
\item \tab \IF $\beta'\neq\mathsf{none}$:
\item \tab\tab $\beta:=\pi^{-B}\beta'$; $f:=f(\pi^{e_2}z+\beta)-\beta g(\pi^{e_2}z+\beta)$;
$g:=\pi^{e_2} g(\pi^{e_2}z+\beta)$
\item \tab\tab $\beta_\total:=\beta_\total+\pi^{e_{2,\total}}\beta$; $e_{2,\total}:=e_{2,\total}+e_2$,
\item \tab\tab \GOTO step 2).
\item \RETURN $(e_{1,\total},e_{2,\total},\beta_\total)$, $(f,g)$.
\end{enumerate}
\end{algorithm}

\section{Determining minimal models over principal ideal domains}

With Algorithm~\ref{A:localminimalmodel} in place, we can turn the procedure sketched in the proof of Corollary~\ref{C:minimal_model} into an algorithm as well. In this section, let $R$ be a principal ideal domain with field of fractions $K$. For a prime ideal $\fp$ we write $R_\fp$ for the localization of $R$ at $\fp$ (we do not need a completion for our purposes). We write $k_\fp$ for its residue class field $R/\fp$. As a uniformizer in $R_\fp$ we choose a generator $\pi\in R$ of $\fp=\pi R$. Furthermore, when we need representatives of $k_\fp$ in $R_\fp$, we assume that we take elements from $R$.

\begin{algorithm}\label{A:globalminimalmodel}
$\mathsf{MinimalModel}(f_\In,g_\In)$
\begin{description}
\itemsep=0.2em
\item[Input] $f_\In,g_\In\in R[z]$ with $\max(\deg(f_\In),\deg(g_\In))=d$ and $\phi=f_\In/g_\In\in \Rat_d(K)$.
\item[Output] $\lambda_\total,\alpha_\total,\beta_\total\in K$ and $f,g\in R[z]$ with
\[(\lambda,A)=(\lambda, \mat{\alpha&\beta\\0&1})\]
such that $[f,g]=[f_\In,g_\In]^{(\lambda,A)}$ is an $R$-minimal model of $[\phi]\in\calM_d(K)$.
If $[f_\In,g_\In]$ is already minimal then $[f,g]=[f_\In,g_\In]$ and $(\lambda_\total,\alpha_\total,\beta_\total)=(1,1,0)$.
\end{description}
\medskip
\begin{enumerate}[\hspace{2em}1)\;]
 \itemsep=0.2em
\item $f,g:=f_\In,g_\In$.
\item Compute the prime factorization $\fp_1^{\epsilon_1},\ldots,\fp_r^{\epsilon_r}=(\Res_d(f,g))R$.
\item \FOR $\fp\in\{\fp_i : i = 1,\ldots,r \text{ and }\epsilon_i\geq d\gcd(2,d+1)\}:$
\item \tab Determine $\pi\in R$ such that $\fp=\pi R$ and choose representatives of $k_\fp$ in $R$.
\item \tab $(e_1,e_2,\beta),(f,g):=\mathsf{LocalMinimalModel}(f,g)$ with respect to $R_\fp$.
\item \tab $\lambda_\total=\lambda_\total\pi^{e_1}$; $\beta_\total:=\beta_\total+\alpha_\total\beta$; $\alpha_\total=\alpha_\total\pi^{e_2}$
\item \RETURN $(\lambda_\total,\alpha_\total,\beta_\total)$, $(f,g)$.
\end{enumerate}
\end{algorithm}
Note that in step 3) we use Lemma~\ref{L:trivial_minimal} to reduce the set of primes to consider. Furthermore, in step 4) we take care to choose $\pi$ and representatives of $k_\fp$ such that the transformation computed to ensure $R_\fp$-minimality in step 5) does not affect the minimality at any other primes. That means we can simply compose the transformations to obtain one that transforms the given model into an $R$-minimal one.

\section{A counterexample to some dynamical analogue of Szpiro's conjecture}
\label{S:szpiro}

In an attempt to formulate a dynamical analogue of Szpiro's conjecture, Silverman suggests the following definition of \emph{conductor} \cite{ADS}*{Section 4.11}.
\begin{definition} Let $R$ be a Dedekind domain with field of fractions $K$. For $\phi\in \Rat_d(K)$ we define
\[\Cond_R([\phi])=\sqrt{\Res_R([\phi])},\]
where $\sqrt{I}$ denotes the \emph{radical ideal} of $I$.
\end{definition}
One analogue of Szpiro's conjecture \cite{ADS}*{Conjecture~4.97} would predict the existence of a bound $n$ and an ideal $J\subset R$ such that
\[J\,\Cond_R([\phi])^n\subset \Res_R([\phi])\text{ for all }\phi\in \Rat_d(K).\]
If $d\geq 3$ and $h(x)\in R[x]$ is a monic polynomial of degree at most $\frac{1}{2}(d-2)$ and $\pi\in R$ such that $J\notin \pi R$, we see that the rational function
\[\phi(x)=\frac{x^d+\pi^{n+1}}{h(x)x}\]
is a counter example, since the given model is locally minimal at all places of $R$ by Remark~\ref{R:trivial_minimal} and therefore globally minimal, but $\Res_d(x^d+\pi^{n+1},h(x)x)$ is divisible by $\pi^{n+1}$.
See also \cite{STW:resultant_conductor} for counterexamples with $d=2$ and an in-depth treatment of possible alternative formulations of the concept of \emph{conductor}. The same paper also discusses some approaches to proving that certain models are minimal. In their Section~3 they consider an approach similar to the valuation-based part of Section~\ref{S:local}. Indeed, without a systematic method for determining possible values for $\beta$ (the utility of Algorithm~\ref{A:inequalitysolutions}), they conclude that their methods are likely insufficient in general. However, in their Section~5 they present some methods based on explicit models for the moduli space $\calM_d$ and its higher level covers. When these work, they likely provide an elegant alternative to Algorithm~\ref{A:localminimalmodel}, although for large $d$ such models might be hard to compute.

\section{The structure of the set of minimal models of a map}
\label{S:mintran}

Let $R$ be a principal ideal domain with field of fractions $K$ and let $\phi\in \Rat_d(K)$. Proposition~\ref{C:minimal_model} guarantees the existence of an $R$-minimal model $[F,G]\in M_d(K)$ for $[\phi]$ and Algorithm~\ref{A:globalminimalmodel} provides a procedure to compute one, given a sufficiently explicit description of $\phi$. In this section we consider the set of all such models
\[\Min_R([\phi])=\{[F,G]: F,G\in R[X,Y] \text{ and }[F,G] \text{ is an $R$-minimal model for }[\phi]\}.\]
It is immediate that $\Min_R([\phi])$ is stable under the action of
$(\GG_m\times \GL_2)(R)$. It can be bigger than a single orbit, as the following example shows.
\begin{example}\label{E:mult_orbit}
Let $n$ be a positive integer and suppose that $c\in\ZZ$ is not $0,\pm 1$. Consider the $\ZZ$-model $[F,G]=[z^{2n+1}-c^{n+1},z^n]$. By Remark~\ref{R:trivial_minimal}, the model is $\ZZ$-minimal. Conjugation by the transformation
\[(\lambda,A)=(c^{-n-1},\mat{c&0\\0&1})\in (\GG_m\times\GL_2)(\QQ)\]
yields the model $[F,G]^{(\lambda,A)}=[c^nz^{2n+1}-1,z^n]$, which has the same resultant and hence is also minimal. It is straightforward to check that these two models are not in the same $(\GG_m\times \GL_2)(\ZZ)$-orbit, for instance by verifying that the set of fixed points of $\phi$ has a trivial stabilizer in $\PGL_2(\QQ)$ and noting that the given transformation does \emph{not} map to $\PGL_2(\ZZ)$.
\end{example}

Note that the rational function in Example~\ref{E:mult_orbit} is of degree $2n+1$, which is odd.

\begin{question} Does there exist a rational function $\phi\in\Rat_d(\QQ)$ with $d$ even, such that $\Min_R([\phi])$ consists of a single $(\GG_m\times\GL_2)(\ZZ)$-orbit?
\end{question}

If $[\phi]$ admits a minimal model $[F,G]$, we can consider the set of transformations
\[\mathrm{MinTran}_R([F,G])=\{(\lambda,A)\in(\GG_m\times\GL_2)(K): [F,G]^{(\lambda,A)}\in\Min_R([\phi])\}.\]
As remarked, this set can be decomposed as a union of left cosets of $(\GG_m\times \GL_2)(R)$. We make some basic observations on the number of cosets.

\begin{proposition}\label{P:mintran_local}
Let $R$ be a discrete valuation ring with field of fractions $K$ and suppose that $[F,G]\in M_d(K)$ is an $R$-model with $\Res_d(F,G)\in R^\times$. Then
\[\mathrm{MinTran}_R([F,G])=(\GG_m\times \GL_2)(R).\]
\end{proposition}

\begin{proof}
Let us assume that $[F,G]$ is a model as given and that $(\lambda,A)\in\mathrm{MinTran}_R([F,G])$. We will show that $\lambda\in R^\times$ and $A\in \GL_2(R)$.

First we show that we can assume that the leading coefficients $f_d$ and $g_d$ are units in $R^\times$. We consider the reduction $\overline{F},\overline{G}\in k[X,Y]$. Our resultant condition implies that $[\overline{F},\overline{G}]\in M_d(k)$. We write $\overline{\phi}$ for the corresponding rational function. We have $f_d,g_d\in R^\times$ if and only if $\overline{\phi}(\infty)\notin\{0,\infty\}$. Note that $\overline{\phi}$ has at most $d+1$ fixed points, so if $\#k > d$ then there are points in $P,Q\in\PP^1(k)$ such that $P$ is not a fixed point and $\overline{\phi}(P)\neq Q$.
We can find a transformation $\overline{T}\in\GL_2(k)$ such that $\overline{T}(\infty)=P$ and $\overline{T}(0)=Q$. We lift $\overline{T}$ to $T\in\GL_2(R)$. It follows that $T^{-1}\phi T$ has the desired property. Since $A\in \GL_2(R)$ if and only if $TA\in\GL_2(R)$, we can restrict to $f_d,g_d$ being units, provided $\#k>d$. However, writing $R^\mathrm{unr}$ for an unramified extension of $R$, we have that $\GL_2(R)=\GL_2(R^\textrm{unr})\cap\GL_2(K)$, so it is sufficient to prove the statement for a sufficiently large unramified extension of $R$. This means we can assume that $\# k$ is sufficiently large and hence that $f_d,g_d\in R^\times$.

We can adapt the results in Section~\ref{S:local}
to determine minimality-preserving transformations by changing the inequalities in \eqref{E:ineqs} to equalities. The argument for Proposition~\ref{P:localminimal} allows us to assume that the transformation is of the form
\[(\lambda,A)=(\pi^{e_1},\mat{\pi^{e_2}&\beta\\0&1})\in (\GG_m\times\GL_2)(K).\]
The claim follows if we can show that $e_2=0$ and $\beta\in R$, since then obviously $e_1=0$.
Indeed from Lemma~\ref{L:e2_bounds} we obtain that $e_2=0$ and from  Lemma~\ref{L:beta_lower_bound} we find that $\beta\in R$. This proves the proposition.
\end{proof}

\begin{example}\label{E:no_affine_minimal_model}
Let $\alpha=\sqrt{-5}$, let $R=\ZZ[\alpha]$ and let $K=\QQ(\alpha)$. We consider
\[\psi(z)=z^2\in K(z).\]
Since $\Res_2(z^2,1)=1$, we see that $\psi$ is minimal and Proposition~\ref{P:mintran_local} yields that
\begin{equation}\label{E:mintran_trivial}
\mathrm{MinTran}_R([z^2,1])\subset\bigcap_{\text{all primes }\fp}(\GG_m\times\GL_2)(R_\fp)=(\GG_m\times\GL_2)(R).
\end{equation}

We consider
\[M=\mat{2&1\\1+\alpha&1} \text{ and } \phi(z)=M\circ\psi\circ M^{-1}=\frac{2z^2+(2\alpha-2)z-\alpha-1}{3z^2+(2\alpha-4)z-\alpha}.\]
We claim that $[\phi]_1$ does not have an $R$-affine minimal model, whereas of course $[\phi]$ does have the $R$-minimal model $[z^2,1]$.
This shows that a non-trivial class group for $R$ can obstruct obtaining affine minimal models even in the presence of a minimal model.

Suppose that $A=\mat{a&b\\0&d}\in\Aff_2(K)$ such that $A\circ\phi\circ A^{-1}$ is represented by an $R$-minimal model. Then
$A^{-1}M^{-1}\circ\psi\circ MA$ is represented by an $R$-minimal model, so \eqref{E:mintran_trivial} yields
\[MA=\mat{2&1\\1+\alpha&1}\mat{a&b\\0&d}=\mat{2a&2b+d\\(1+\alpha)a&(1+\alpha)b+d}\in\GL_2(R).\]
Since the ideal $\fp_2=2R+(1+\alpha)R$ is of norm $2$ and non-principal, we see that $2a,(1+\alpha)a\in R$ implies that $a\in R$. But then $\det(MA)\in \fp_2$, which contradicts that $MA\in\GL_2(R)$.
\end{example}

\begin{proposition} Let $K$ be a global field and suppose that its ring of integers $R$ is a principal ideal domain. Let $[F,G]\in M_d(K)$ be an $R$-minimal model for $[\phi]\in\calM_d(K)$. Then
\[\mathrm{MinTran}_R([F,G])\]
is a finite union of left-cosets of $(\GG_m\times\GL_2)(R)$.
\end{proposition}
\begin{proof}
We have to establish that a finite union suffices. Let $S$ be the finite set of places where $\Res_d(F,G)$ is not a unit. We write $R_S$ for the ring of $S$-integers. Since $K$ is a \emph{global} field, we have that all residue fields are finite and hence that $R_S^\times$ is finitely generated.
Proposition~\ref{P:aff_decomp} shows that each coset has a representative in $(\GG_m\times\Aff_2)(K)$ and Proposition~\ref{P:mintran_local} shows that we can take the representatives of the form
\[
(\lambda,\mat{\alpha&\beta\\0&1}),
\]
where $\lambda=1/\alpha$ and $\alpha$ is an $S$-unit. Note that Lemma~\ref{L:e2_bounds} provides us with valuation bounds on $\alpha$ and that the coset represented only depends on the value of $\alpha$ in $R_S^\times/R^\times$. Therefore, we only have to consider finitely many representatives for $\alpha$.

Similarly, for $\beta$ we have that Lemma~\ref{L:beta_lower_bound} provides lower bounds on the valuations of $\beta$ and that the coset represented only depends on the value of $\beta$ in $K/\alpha^{-1}R$, which only leaves us with finitely many candidates.
\end{proof}

\begin{remark} Note that $R_S^\times/R$ is also finitely generated if the residue fields of $R$ are not finite. We only use that $K$ is global for establishing that finitely many representatives for $\beta$ suffice. However, note that the lower bounds provided by Lemma~\ref{L:beta_lower_bound} only give \emph{necessary} conditions. It may well be that the full problem \eqref{E:V} is so restrictive that any solution would lead to one of finitely many cosets regardless of the finiteness of the residue field. As a concrete question, one may ask:
\end{remark}

\begin{question} Let $k$ be a field, let $R=k[[t]]$ be the ring of formal power series and let $K=k(\!(t)\!)$ be the corresponding field of Laurent series. Does there exist a minimal model $[F,G]\in M_d(K)$ such that $\mathrm{MinTran}_R([F,G])$ is not a finite union of left cosets of $(\GG_m\times \GL_2)(R)$?
\end{question}

\section{Orbits of rational functions containing many integer points}\label{S:intpoints}

In this section we restrict to $R=\ZZ$ and $K=\QQ$. In order to obtain a concept of integrality on $\PP^1(\QQ)$, we fix a point at infinity and consider $\ZZ\subset\QQ\subset\PP^1(\QQ)=\QQ\cup\{\infty\}$. Let $\phi\in\QQ(z)$ be a rational map on $\PP^1(\QQ)$. For a point $\alpha\in\PP^1(\QQ)$ we consider the \emph{forward orbit}
\[\calO_\phi(\alpha)=\{\alpha,\phi(\alpha),\phi^2(\alpha),\ldots\},\]
where $\phi^k=\phi\circ\cdots\circ\phi$ means composition of $\phi$ with itself. We say that $\alpha$ is a \emph{wandering point} if $\calO_\phi(\alpha)$ is an infinite set. In direct analogy with Siegel's theorem that a curve of genus one has only finitely many integral points, we have 

\begin{theorem}[\cite{sil:intpoints}*{Theorem A}, \cite{ADS}*{Theorem 3.43}]
Let $\phi(z)\in\QQ(z)$ be a rational map of degree $d\geq 2$ such that $\phi^2(z)\notin\QQ[z]$. Let $\alpha\in\QQ$ be a wandering point for $\phi$. Then $\calO_\phi(\alpha)$ contains only finitely many integer points.
\end{theorem}

The following example shows that, just as elliptic curves can have arbitrarily many integer points (see for instance \cite{mahler:latticepoints}), we can construct rational maps with arbitrarily many integer points in their orbits too.

\begin{example}[See \cite{ADS}*{Example~3.45}] Let $\phi(z)=(z^2+z+1)/(z^2-z+1)$. Then $\calO_\phi(0)=\{0,1,3,13/7,\ldots\}$. We can construct another rational map with more integer points in its orbit by scaling the denominator out. Consider $\psi(z)=7\phi(z/7)$ with $\calO_\psi(0)=\{0,7,21,13,2163/127,\ldots\}$. We can iteratively scale out consecutive denominators and construct rational functions with arbitrarily many integer points in their orbits.
\end{example}

In the example above we have $[\phi]=[\psi]\in\calM_2(\QQ)$. The associated models have
$\Res_2(z^2+z+1,z^2-z+1)=4$ and $\Res_2(7x^2 + 49x + 343, x^2 - 7x + 49)=4\cdot7^6$, so the function obtained by scaling is not given by a minimal model.

Analogous to a conjecture by Dem'janenko-Lang \cite{lang:ec_da}*{p.\ 140} on uniform bounds on the number of integral points on minimal Weierstrass models of elliptic curves, Silverman conjectures

\begin{conjecture}[\cite{ADS}*{Conjecture~ 3.47}] For $d\geq 2$ there is a constant $C_d$ such that for any rational map $\phi\in\Rat_d(\QQ)$ such that $\phi^2$ is not a polynomial given by a model $[F,G]\in M_d(\QQ)$ that is $\ZZ$-minimal for $[\phi]\in\calM_d(\QQ)$ and any wandering point $\alpha$, we have that $\calO_\phi(\alpha)$ contains at most $C_d$ integer points.
\end{conjecture}

Silverman makes a conjecture that is a priori stronger by demanding that $\phi$ is only \emph{affine minimal}, but Proposition~\ref{P:GL_and_affine_minimal} shows that over $\ZZ$ this formulation is equivalent. In \cite{sil:intpoints} he also mentions an example $\phi(z)=(-54z^2+16z+128)/(z^2-41z+64)$ for which $\calO_\phi(0)$ contains at least $7$ integer values. Unfortunately, $\psi(z)=\phi(8z)/8=(-54z^2 + 2 + 2)/(8z^2 - 41z + 8)$ has a smaller resultant, so $\phi$ is not (affine) minimal.

In the same paper Silverman also mentions that it would be interesting to exhibit minimal rational functions of degree 2 with at least $8$ integer points in an orbit. We describe one approach to finding such functions.

First we remark that a simple interpolation argument shows that a sufficiently long initial part of a wandering orbit determines a rational function of given degree uniquely. Suppose that $\phi(z)=f(z)/g(z)$ is a rational function of degree $d$ with orbit $\{c_0,\ldots,c_r,\ldots\}$. Then the coefficients $f_d,\ldots,f_0,g_d,\ldots,g_0$ satisfy the linear system
\begin{equation}\label{E:c_system}
\mat{c_0^d&\cdots&1&-c_1c_0^d&\cdots&-c_1\\
 c_1^d&\cdots&1&-c_2c_1^d&\cdots&-c_2\\
\vdots&&\vdots&\vdots&&\vdots\\
 c_{r-1}^d&\cdots&1&-c_{r}c_{r-1}^d&\cdots&-c_{r}}
\mat{f_d\\\vdots\\f_0\\g_d\\\vdots\\g_0}=0.
\end{equation}
Indeed, setting $c_0=0$, the affine plane $\AA^{2d+1}$ with coordinates $(c_1,\ldots,c_{2d+1})$ is birational to $\Rat_d$. There are some obvious loci on which this birationality is not defined. For instance, when $c_i=c_j$ for $i\neq j$ or when a significant part of the orbit already fits a lower degree function, for example $d=2$ and $(c_1,\ldots,c_5)=(1,3,7,15,c_5)$.

In particular, we see that in order for $\{c_0,\ldots,c_{2d+2}\}$ to be an orbit of a degree $d$ function, the matrix in \eqref{E:c_system} must have determinant $0$. This leads to a relation
\[N(c_0,\ldots,c_{2d+1})-c_{2d+2} D(c_0,\ldots,c_{2d+1})=0\text{ with }
N,D\in\ZZ[c_0,\ldots,c_{2d+1}.]\]
Furthermore, $N$ is of total degree $(d+1)^2$ and $D$ is of total degree $d(d+2)$.
Both $N$ and $D$ are of degree $d+1$ in each of $c_1,\ldots,c_{2d+1}$ and of degree $d$ in $c_0$.

A reasonable strategy to find rational maps with an orbit containing many integers is now to set a bound $B>0$, choose $c_0,\ldots,c_{2d+1}\in\{-B,\ldots,B\}$ and see for which values we have that $D(c_0,\ldots,c_{2d+1})$ divides $N(c_0,\ldots,c_{2d+1})$. To reduce the search we can restrict to $c_0=0$ and $c_1>0$. For each of the found vectors $(0,c_1,\ldots,c_{2d+2})$ we check if there is indeed a corresponding degree $d$ rational function and whether the resulting model is minimal using Algorithm~\ref{A:globalminimalmodel}.

For $d=2$ it turns out that $N$ has $70$ monomials and largest coefficient $4$ and $D$ has $76$ monomials with largest coefficient $3$. Since $76\cdot3\cdot 100^8 < 2^{63}$ and $4\cdot 100^9 < 2^{63}$, we can take $B=100$ and do the divisibility test with word-sized integers on a $64$-bit machine, provided we reduce the terms of $N$ modulo the value of $D$ before adding them. This approach allowed us to test the roughly $1.5\cdot10^{11}$ candidates with $c_1\in\{1,\ldots,100\}$ and $c_2,\ldots,c_5\in\{-100,\ldots,100\}$ in about 4 days on a 2.33GHz machine. We used Cython \cite{cython} and Sage \cite{sage} for the implementation of the computer program. Our findings are summarized in Table~\ref{T:deg2}. A full list of orbits found is available electronically from \cite{brumol:electronic}.

In order to prove that the orbit of $0$ is indeed infinite we make use of the following result.
\begin{theorem}[\cite{ADS}*{Theorem~2.21} or  \cite{morsil:perpoints}*{Theorem~1.1}]\label{T:torbound}
Let $\phi\in \Rat_d(\QQ)$ and let $[F,G]$ be a model of $\phi$ over $\ZZ$. Let $p$ be a prime not dividing  $\Res_d(F,G)$. Then there is an explicit procedure to produce a finite set $M(\phi,p)$ such that for any $\alpha\in\PP^1(\QQ)$ such that $\alpha$ is a periodic point under $\phi$, we have that
\[\phi^k(\alpha)=\alpha \text{ for some }k\in\{mp^e: m \in M(\phi,p),\;e\in\{0,1,\ldots\}\}.\]
The construction guarantees that no element of $M(\phi,p)$ is divisible by a prime bigger than $p+1$.
\end{theorem}

A consequence of this theorem is that if $p_0\geq 3$ is a prime of good reduction for $\phi$, then no primes bigger than $p_0$ will divide the period of any periodic rational point, so if we take good primes $3\leq p_0<p_1<\cdots<p_r$ and compute
\[ M=\bigcap_{i=1}^r M(\phi,p_i),\]
then any point $\alpha\in\PP^1(\QQ)$ periodic under $\phi$ is a solution to $\phi^k(z)=z$ for some $k\in M$. The nature of the explicit procedure yields that $M$ is likely very small for even small values of $r>2$, so one can find all rational periodic points by solving a finite and likely small number of polynomial equations.
We can find all rational preperiodic points by computing the rational points in the inverse orbits of the periodic points. This is a matter of iteratively solving equations of the form $\phi(z)=\alpha$ for appropriate $\alpha$.
We can check that $0$ is a wandering point by verifying it does not occur in the list of preperiodic points we construct above. See \cite{brumol:electronic} for an implementation of this procedure.

For each of the $2190$ minimal rational functions for which the initial $7$ members of the orbit of $0$ are integral, we checked whether there are any further integers early in the orbit. We found $4$ functions where the orbit starts with $8$ integers and a fifth function with $8$ integers, but not in consecutive spots. See Table~\ref{T:deg2examples}.

We also used this strategy to find degree $3$ rational functions with many integers in the orbit of $0$. Using the same approach as for degree $2$ functions, we find that we can prescribe orbits  $[0,c_1,\ldots,c_7]$ with $c_1\in\{1,\ldots,10\}$ and $c_2,\ldots,c_7\in\{-10,\ldots,10\}$. Again, we can express $c_8=N(c_1,\ldots,c_7)/D(c_1,\ldots,c_7)$, where $N$ has total degree $16$ and $D$ has total degree $15$. Searching through tuples of distinct integers $(c_1,\ldots,c_7)$ in this range such that $D(c_1,\ldots,c_7)$ divides $N(c_1,\ldots,c_7)$ took about 31 hours. Again, we check the resulting tuples for minimality, polynomials and preperiodic orbits. Our findings are summarized in Table~\ref{T:deg3}. See \cite{brumol:electronic} for all found orbits.

For each of the $6508$ resulting functions we found that $28$ functions had a tenth integer in the orbit of $0$ and $25$ functions had an integer preimage for $0$. However, $11$ of these are translates of other functions, so we find $42$ minimal degree $3$ functions with at least $10$ integers consecutively in an orbit. 
We also found $6$ examples where a tenth integer point occurred after a non-integral or an infinite value. See  \cite{brumol:electronic} for a full list and Table~\ref{T:deg3examples} for a small sample.

\section*{Acknowledgements}

We are particularly grateful for the generous support of the 2010 Arizona Winter School and in particular Joseph Silverman who carefully prepared lectures and assignments. The work described here drew inspiration from one of the problems he set for the school.
We also thank an anonymous referee for inquiring about examples in non-principal domains, which led to a strengthening of Proposition~\ref{P:GL_and_affine_minimal} and Example~\ref{E:no_affine_minimal_model}.

\begin{table}
\[
\begin{array}{c|l}
\phi(z)&\calO_\phi(0)\\
\hline
\dfrac{86z^2 - 1068z - 338}{z^2 + 7z - 338}&[0,\, 1,\, 4,\, 11,\, 12,\, 7,\,15,\,-374,\, \ldots]\\[0.5em]
\dfrac{-61z^2 - 1279z + 1862}{4z^2 + 114z + 266}&[0,\, 7,\, -8,\, -21,\, -5,\, -33,\,-26,\,-1020,\, \ldots]\\[0.5em]
\dfrac{25z^2 - 1895z - 8910}{58z^2 - 146z - 990}&[0,\, 9,\, -10,\, 2,\, 12,\, -5,\,1,\,10,\, \ldots]\\[0.5em]
\dfrac{367z^2 - 15104z + 143325}{12z^2 - 469z + 4095}&[0,\, 35,\, 27,\, 17,\, 18,\, 21,\,26,\,-99,\, \ldots]\\[0.5em]
\dfrac{12z^2 - 29z - 35}{z^2 + 8z - 35}&[0,\, 1,\, 2,\, 3,\, 7,\, 5,\,4,\,\frac{41}{13},\,-40,\,\ldots]
\end{array}
\]
\caption{Some explicit degree $2$ functions with $8$ integers in an orbit}\label{T:deg2examples}
\begin{center}
\begin{tabular}{| p{40ex} | r | }
  \hline
  Size of search space &  150 617 612 376\\ \hline
  Orbits with a $7$-th integer point & 2 112 933 \\ \hline
  Orbits corresponding to minimal maps & 2 261 \\ \hline
  Preperiodic orbits & 64\\ \hline
  Polynomials & 7\\ \hline
  Non-polynomial, infinite orbits with at least 7 integer points in the orbit of 0& 2 190\\
  \hline
\end{tabular}
\end{center}
\caption{Search results for rational functions of degree 2 with many integers in the orbit of 0}\label{T:deg2}
\[
\begin{array}{c|l}
\phi(z)&\calO_\phi(0)\\
\hline
\dfrac{7z^3 - 41z^2 - 216z + 180}{2z^3 - z^2 - 21z + 90}&[0,\,2,\,-6,\,6,\,-3,\,3,\,-9,\,5,\,-5,\,8,\ldots]\\[0.5em]
\dfrac{-6z^3 - 10z^2 + 29z - 3}{z^3 - 8z - 3}&[0,\, 1,\, -1,\, -9,\, -5,\, -4,\, -3,\, 3,\,\infty,\,-6,\ldots]\\[0.5em]
\dfrac{35z^3 - 219z^2 + 292z + 60}{5z^3 - 18z^2 - 26z + 60}&[0,\, 1,\, 8,\, 5,\, 4,\, 3,\, 2,\, -2,\,\infty,\,7,\ldots]\\[0.5em]
\dfrac{-24z^3 + 285z^2 - 825z + 252}{z^3 + 15z^2 - 142z + 126}&[0,\, 2,\, 5,\, -3,\, 9,\, -2,\, 7,\, 1,\,\infty,\,-24,\ldots]\\[0.5em]
\end{array}
\]
\caption{Some explicit degree $3$ rational functions with $10$ integers in an orbit}\label{T:deg3examples}
  \begin{center}
    \begin{tabular}{| p{45ex} | r | }
     \hline
    Size of search space & 195 350 400 \\ \hline
    Orbits with a $9$-th integer point & 44 563 \\ \hline
    Orbits belonging to minimal maps & 7 631 \\ \hline
    Orbits corresponding to non-degree 3 maps & 3\\ \hline
    Degree 3 polynomial orbits & 0\\ \hline
    Degree 3, preperiodic orbits & 913\\ \hline
    Degree 3 non-preperiodic, orbits with at least 9 integer points in the orbit of 0& 6 508\\ 
    \hline
    \end{tabular}
  \end{center}
\caption{Degree 3 functions with many integer points in the orbit of $0$}\label{T:deg3}
\end{table}

\end{document}